\documentclass[12pt,a4paper]{amsart}

\usepackage{amssymb,amscd}
\usepackage{mathrsfs}
\usepackage[mathcal]{eucal}
\usepackage{float}
\usepackage[stable]{footmisc}

\makeatletter
\newcommand{\dotminus}{\mathbin{\text{\@dotminus}}}

\newcommand{\@dotminus}{%
  \ooalign{\hidewidth\raise1ex\hbox{.}\hidewidth\cr$\m@th-$\cr}%
}
\makeatother

\usepackage{xy}
\input xy
\xyoption{all}
\usepackage{pdflscape}
\usepackage{hyperref}

\usepackage[english]{babel}

\let\myacute=\'

\def\<{\langle}
\def\>{\rangle}

\def\Z{\mathbb{Z}}

\def \begindm {\begin{displaymath}}

\def \enddm {\end{displaymath}}

\newtheorem{thm}{Theorem}[section]
\newtheorem{lemma}{Lemma}[section]
\newtheorem{cor}{Corollary}[section]

\numberwithin{equation}{section}

\long\def\symbolfootnote[#1]#2{\begingroup\def\thefootnote{\fnsymbol{footnote}}\footnote[#1]{#2}\endgroup}

\title{Truncations of Ordered Abelian Groups}

\author[P. D'Aquino]{Paola D'Aquino} 
\address{Dipartimento di Matematica, Universit\`{a} della Campania "L. Vanvitelli", viale Lincoln, 5, 81100
Caserta, Italy}
\email{paola.daquino@unicampania.it}

\author[J. Derakhshan]{Jamshid Derakhshan}
\address{St Hilda's College, University of Oxford, Cowley Place, Oxford OX4 1DY, UK}
\email{derakhsh@maths.ox.ac.uk}

\author[A. Macintyre]{Angus Macintyre${}^{\dag}$}
\address{School of Mathematical Sciences, Queen Mary, University of London, Mile End Road, London E1 4NS, and 
School of Mathematics, University of Edinburgh, King's Buidings, Peter Guthrie Tait Road
Edinburgh EH9 3FD, UK}
\email{angus@eecs.qmul.ac.uk}

\thanks{${}^{\dag}$Supported by a Leverhulme Emeritus Fellowship}

\begin{document}


\subjclass[2000]{}



\maketitle

\section{\bf Introduction}\label{sec-introduction}

This paper is a component of two pieces of research, one by D'Aquino and Macintyre \cite{PDAJM},\cite{elem-prod}, and one by Derakhshan and  Macintyre \cite{elem-rest}. The common theme is the model theory of local rings $V/(\alpha)$, where $V$ is a Henselian valuation domain, and $\alpha$ is a nonzero nonunit of $V$. If $v: V\rightarrow P$ is the valuation of $V$, with $P$ the semigroup of non-negative elements of the value group $\Gamma$ of the fraction field $K$ of $V$, $v$ induces a ''truncated valuation'' from $V/(\alpha)$ onto the segment $[0,v(\alpha)]$ of $P$ defined by $v(x+(\alpha))=v(x)$ if $v(x)<v(\alpha)$, and $v(x+(\alpha))=v(\alpha)$ if $v(x)\geq v(\alpha)$. 

The segment $[0,v(\alpha)]$ inherits from $\Gamma$ an ordering $\leq$, with $0$ as least element and $v(\alpha)$ as greatest element. From our assumption on $\alpha$, $v(\alpha)\neq 0$.

Next, $[0,v(\alpha)]$ gets a truncated semi-group structure as follows. Let $\oplus$ be the addition on $\Gamma$. Define, for 
$\gamma_1,\gamma_2\in [0,v(\alpha)]$
$$\gamma_1+\gamma_1=\mathrm{min}(\gamma_1\oplus\gamma_2,v(\alpha)).$$
The basic laws are
$$v(x+y)\geq \mathrm{min}(v(x),v(y))$$
$$v(xy)=\mathrm{min}(v(\alpha),v(x)+v(y)).$$
Ideas connected to this, but much more sophisticated, appear in the work of Hiranouchi \cite{hiranouchi-survey},\cite{hiranouchi-ext}.

Forgetting the valuation in the preceding, we have an ordered abelian group $\Gamma$, with order $\leq$, addition $\oplus$, subtraction $\ominus$, zero $0$, and a distinguished element $\tau>0$ (where $\tau=v(\alpha)$ in the preceding). We then define as above a "truncated addition" $+$ on $[0,\tau]$, giving us an example of a truncated ordered abelian group (TOAG). 

Our main result produces a natural first-order set of axioms in the language $\{\leq,0,\tau,+\}$ for truncated ordered abelian groups (these axioms are true in initial segments of ordered abelian groups), and proves that each truncated ordered abelian group (i.e. each model of these axioms) is an initial segment of an ordered abelian group.

\section{\bf The axioms}

\

1.2. {\it Obvious axioms.} The following are obvious, via immediate calculations in $P$.

{\it Axiom} 1. Addition $+$ is commutative.

\

{\it Axiom} 2. $x+0=x$.

\

{\it Axiom} 3. $x+\tau=\tau$.

\

{\it Axiom} 4. If $x_1\leq y_1$ and $x_2\leq y_2$ then $x_1+x_2\leq y_1+y_2$.

\

1.3. {\it Less obvious axioms.}

\

{\it Axiom} 5. Addition $+$ is associative.

\

{\it Verification.} Suppose $x,y,z$ in $[0,\tau]$.

\

{\it Case} 1. $x \oplus y \oplus z \leq \tau$. Then
$$x+(y+z)=x+(y\oplus z)=x\oplus (y\oplus z)=(x\oplus y)\oplus z=(x+y)+z$$

\

{\it Case} 2. $x\oplus y \oplus z >\tau$.

\

{\it Subcase} 1. $y\oplus z\geq \tau$ and $x\oplus y \geq \tau$. Then 
$$x+(y+z)=x+\tau=\tau$$
$$(x+y)+z=\tau+z=\tau.$$

\

{\it Subcase} 2. $y\oplus z <\tau, x\oplus y \geq \tau$. 

Then 
$$x+(y+z)=x+(y\oplus z)=min(x\oplus (y\oplus z),\tau)=min(x\oplus y\oplus z,\tau)=\tau,$$
$$(x+y)+z=\tau+z=\tau.$$

\

{\it Subcase} 3. $y\oplus z\geq \tau$ and $(x\oplus y)<\tau$. 

Then
$$x+(y+z)=x+\tau=\tau$$
$$(x+y)+z=(x\oplus y)+z=min((x\oplus y)\oplus z,\tau)=min(x\oplus y\oplus z,\tau)=\tau.$$

\

{\it Subcase} 4. $y\oplus z<\tau,~ x\oplus y<\tau$.

\

Now $x\oplus (y\oplus z)=(x\oplus y)\oplus z$ and so $x\oplus(y+ z)=(x+ y)\oplus z$ and so
$$min(x\oplus (y+z),\tau)=min((x+y)\oplus z,\tau),$$
so $x+(y+z)=(x+y)+z$.

\

1.4. {\it Axioms concerning cancellation.}

\

{\it Axiom} 6. If $x+y=x+z <\tau$, then $y=z$.

\

{\it Verification.} $x\oplus y=x+y$ and $x\oplus z=x+z$ in this case, so use cancellation in $P$.

\

{\it Axiom} 7. If $x\leq y<\tau$, then there is a unique $z$ with $x+z=y$.

\

{\it Verification.}  Immediate from definition and the fact that $P$ is the non-negative part of $\Gamma$.

\

{\it Notation.} We write $y \dotminus x$ for the $z$ in the above.

\

{\it Axiom} 8. There are (in general, many) $z$ in $[0,\tau]$, for $x$ in $[0,\tau]$, so that $x+z=\tau$, and there is a minimal one to be denoted $\tau \dotminus x$.

\

{\it Verification.} Obvious by working in $P$ and taking $\tau \dotminus x=\tau -x$.

\

We define $\tau \dotminus \tau$ as $0$. 

\

{\it Axiom} 9. $\tau \dotminus (\tau \dotminus x)=x$.

\

{\it Verification.} For $x<\tau$, 
$$\tau \dotminus (\tau \dotminus x)=\tau - (\tau-x).$$ 
For $x=\tau$, 
$$\tau \dotminus (\tau \dotminus x)=\tau \dotminus 0=\tau=x.$$

\

1.5. {\it Crucial Axioms.} There now follows a series of axioms which are basic in what follows. It is not clear to us what are the dependencies between these axioms over the preceding nine.

\

{\it Axiom} 10. Suppose $0\leq x,~ y<\tau$ and $x+y=\tau$. Then $y\dotminus(\tau \dotminus x)=x\dotminus(\tau \dotminus y)$.

\

{\it Verification.} Both are equal to $(x\oplus y) \ominus \tau$.

\

Before getting to the remaining axioms, we prove some useful lemmas without using Axiom 10.

\

1.6.

\begin{lemma}\label{lem1} Suppose $0\leq y\leq z<\tau$. Then $\tau\dotminus z \leq \tau\dotminus y$.\end{lemma}
\begin{proof} We have that $y+(\tau \dotminus y)=\tau$ and $z+(\tau\dotminus z)=\tau$. 

 If $(\tau \dotminus z) >(\tau \dotminus y)$, then (Axiom 8) 
$$z+(\tau \dotminus y)<\tau,$$
so
$$y+(\tau \dotminus y)<\tau,$$
contradiction.\end{proof}

\begin{lemma}\label{lem2} For $x,y<\tau$, $\tau \dotminus x=\tau\dotminus y$ implies $x=y$.\end{lemma}
\begin{proof} Assume $\tau \dotminus x=\tau\dotminus y$. Then 
$$\tau \dotminus (\tau \dotminus x)=\tau \dotminus (\tau \dotminus y),$$ so (Axiom 9) $x=y$.\end{proof}

\begin{cor}\label{cor1} If $x<y<\tau$, then $\tau\dotminus y<\tau \dotminus x$.\end{cor}
\begin{proof} By Lemma \ref{lem2},
$$\tau \dotminus y \leq \tau \dotminus x.$$
If $\tau \dotminus y=\tau\dotminus x$, then $x=y$.\end{proof}

\

1.7. {\it A miscellany of other axioms.} In the course of proving Associativity in Theorem 1 we need various axioms about $+,\dotminus$, and $\tau$. Each of these axioms is true (and with a trivial proof) in the $[0,\tau]$ coming from the ordered abelian group $\Gamma$ with $\oplus$, so it is natural to use them. One may hope to deduce them from the axioms listed already, but we have not succeeded in doing so. Thus we settle for quite a long list of "ad hoc" axioms, which we now consider in the order in which they occur in the proof of Associativity in 2.8.

\

{\it Axiom} 11. Assume $y+z<\tau$, ~ $x+(y+z)=\tau$, ~ and $y+x<\tau$. Then 

$$x\dotminus(\tau \dotminus(y+z))=z \dotminus(\tau \dotminus(x+y)).$$

\

{\it Note}. By Axiom 5 we may safely write $x+y+z$ for $x+(y+z)$ and $(x+y)+z$ and will do so henceforward.

\

{\it Verification.} In all that follows, we construe $x,y,z,\tau$ as in the $[0,\tau]$ in $(P,<_P,\oplus)$ the non-negative part of an ordered abelian group $\Gamma$, where $<$ is identified with $<_P$ (namely, the restriction of $\leq$ to $P$), and $+$ with the truncation of $\oplus$. Then 
$$x\oplus y\oplus z=\tau \oplus \epsilon,$$
for some $\epsilon$ in $P$. Now $y\oplus z<\tau$ and $y\oplus x<\tau$, so $x\oplus y\oplus z<2\tau$, so $\epsilon < \tau$. 

Let 
$$\mu=\tau-(y+z),$$
$$\delta=\tau-(x+y).$$
Then $x=\mu+\epsilon$, so 
$x\dotminus(\tau\dotminus(y+z))=\epsilon$ and $z=\delta+\epsilon$, so
$$z\dotminus(\tau\dotminus(x+y))=\epsilon,$$
giving the verification.

 The subsequent verifications are at the same level of difficulty.

\

{\it Axiom} 12. Assume $y+z<\tau$, ~ $x+y+z=\tau$, ~ $y+x=\tau$, and $z+(y\dotminus (\tau \dotminus x))<\tau$. Then 
$$z+(y\dotminus (\tau\dotminus x))=x\dotminus (\tau\dotminus (y+z)).$$ 

\

{\it Verification.} (Convention: Any time we write $A\dotminus B$ we assume $A\geq B$). We have 
$$x\oplus y \oplus z=\tau\oplus \epsilon$$ 
for some $\epsilon \in P$, and 
$$y\oplus x=\tau\oplus \gamma,$$
for some $\gamma \in P$. 

Now $x\oplus y \oplus z<2\tau$, so $\epsilon < \tau$; and $y\oplus z<2\tau$, so $\gamma<\tau$.

Let $\mu=\tau-(y+x)$. Then $x=\mu\oplus \epsilon$, so 
$$x\dotminus (\tau\dotminus(y+z))=\epsilon$$
whereas,
$$z+(y\dotminus(\tau\dotminus x))=z+(y-(\tau-x))=x+y+z-\tau=\epsilon.$$

\

{\it Axiom} 13. Assume $y+x=\tau$ and $y+z<\tau$. Then $z+(y\dotminus (\tau\dotminus x))<\tau$.

\

{\it Verification.} Let $y\oplus x=\tau\oplus \epsilon$, where $0\leq \epsilon <\tau$. Then 
$$y\dotminus (\tau \dotminus x)=\epsilon.$$
and 
$$z\oplus \epsilon=z\oplus ((y\oplus x)\ominus \tau)=(x\oplus y\oplus z)\ominus \tau<\tau$$
since $x<\tau$ and $y+z<\tau$.

\

{\it Axiom} 14. Assume $y+z=y+x=\tau$ and $z+(y\dotminus (\tau\dotminus x))<\tau$. Then 
$$x+(y\dotminus (\tau\dotminus z))=(x\dotminus (\tau\dotminus y))+x.$$

\

{\it Verification.} Let 
$$y\oplus z=\tau \oplus \epsilon$$
$$y\oplus x=\tau\oplus \delta,$$
with $0\leq \epsilon,\delta<\tau$. So 
$$y\dotminus (\tau\dotminus z))=\epsilon,$$
$$x\dotminus (\tau\dotminus y)=\delta,$$ 
$$y\oplus z\oplus \delta=\tau\oplus \epsilon \oplus \delta,$$ 
and 
$$y\oplus x\oplus \epsilon=\tau\oplus \epsilon \oplus \delta,$$
so $x+\epsilon=z+\delta$ as required.

\

{\it Axiom} 15. Assume $y+z=\tau$, ~ $y+x=\tau$, and $x+(y\dotminus (\tau\dotminus z))=\tau$. Then 
$z+(y\dotminus (\tau\dotminus x))=\tau$.

\

{\it Verification.} Let
$$y\oplus z=\tau \oplus \delta,$$
$$y\oplus x=\tau\oplus \epsilon.$$
as before. Then 
$$y\dotminus (\tau\dotminus z))=\delta$$
$$y\dotminus (\tau\dotminus x))=\epsilon.$$
$$x\oplus (y\ominus (\tau \ominus z))=(x\oplus y\oplus z)\ominus \tau$$
since
$$x+(y\dotminus (\tau\dotminus z))=\tau$$
$$x\oplus y\oplus z\geq 2\tau$$
Now $z\oplus (y\ominus(\tau\ominus x))=(x\oplus y \oplus z)\ominus \tau$, so $z+(y\dotminus (\tau\dotminus x))=\tau$ since 
$x\oplus y\oplus z\geq 2\tau$.

\

{\it Axiom} 16. Assume 
$$y+z=y+x=x+(y\dotminus (\tau\dotminus z))=\tau.$$ Then
$$(y\dotminus (\tau\dotminus x))\dotminus (\tau\dotminus z)=(y\dotminus (\tau\dot z))\dotminus (\tau\dotminus x).$$

\

{\it Verification.} 
$$(y\dotminus (\tau\dotminus x))\dotminus (\tau \dotminus z)=$$
$$((y\oplus x)\ominus \tau) \ominus (\tau \ominus z)=$$
$$(x\oplus y \oplus z)\ominus (2\tau)=$$
$$((y\oplus z)\ominus \tau)\ominus (\tau \ominus x)=$$
$$(y \dotminus (\tau\dotminus z))\dotminus (\tau\dotminus x).$$

\section{\bf Truncated ordered abelian groups and ordered abelian groups}

2.1. A truncated ordered abelian group (TOAG) is a linear order $[0,\tau]$ with a $+$ satisfying Axioms 1-16.

\

2.2. 

\begin{thm}\label{toag} Let $[0,\tau]$ be a truncated ordered abelian group with $+$ and $\leq$. 
Then there is an ordered abelian group $\Gamma$, under $\oplus$ and $\leq_{\Gamma}$, with $P$ the semigroup of non-negative elements, and an element $\tau_P$ of $P$ so that $[0,\tau]$ 
with $+$ and $\leq$ is isomorphic to $[0,\tau_P]$ with the addition and order induced by $\oplus$ and $\leq_P$ on $[0,\tau_P]$.\end{thm}
\begin{proof} We begin by constructing $P$, and then we laboriously verify that it has the required properties. 

\

2.3. {\it Construction.} $P$ is $\omega \times [0,\tau)$, where $\omega$ is the set of finite ordinals $k$ under ordinal addition ($+$) and order ($\leq$). $[0,\tau)$ has the order induced form $[0,\tau]$. 

$P$ is lexicographically ordered with respect to the two orderings just specified. Let $\leq_P$ be the lexicographic order. 

Let $0=<0,0>\in P$, the least element of $P$. Let $\tau_P=<1,0>\in P$. We have $[0,\tau_P)=\{0\} \times [0,\tau)$, giving natural order isomorphisms $[0,\tau_P)\cong [0,\tau)$ and $[0,\tau_P]\cong [0,\tau]$. 

Now we define $\oplus$ on $P$.

\

2.4. {\it The Case1/Case 2 distinction for $<y,z>\in [0,\tau)^2$.} 

\

This comes up all the time in what follows, and the axioms from Axiom 10 on relate to it.

\

{\it Case 1.} $y+z<\tau$

\

{\it Case 2.} $y+z=\tau$

\

Note that both are symmetric in $y$ and $z$ and are complements of each other.

The main point is that Case 2 is equivalent to both $y\geq \tau \dotminus z$ and $z\geq \tau \dotminus y$. 

Moreover Axiom 10 implies that in Case 2 
$$y\dotminus (\tau\dotminus z)=z\dotminus (\tau\dotminus y).$$

\

2.5. {\it Defining $\oplus$.} 

\

We define $<k,y> \oplus <l,z>$ to be $<k+l,y+z>$ if $<y,z>$ in Case 1, and $<k+l+1,y\dotminus (\tau \dotminus z)>$ if $<y,z>$ in Case 2.

\

2.6.

\begin{lemma} $\oplus$ is commutative, with $0$ as neutral element.\end{lemma}
\begin{proof} Let $<k,y>,<l,z>\in P$. 

\

{\it Case} 1. If $<y,z>$ in Case 1 (and so then is $<z,y>$). Then 
$$(<k,y>\oplus <l,z>=<k+l,y+z>=<l+k,z+y>=<l,z>\oplus <k,y>.$$

\ 

{\it Case} 2. If $<y,z>$ in Case 2 (again symmetric). Then
$$<k,y>\oplus <l,z>=<k+l+1,y\dotminus (\tau \dotminus z)>=<l+k+1,z \dotminus (\tau \dotminus y)>,$$
by Axiom 10.

If $y=0$ we are in Case 1 so $<k,y>\oplus <l,z>=<k+l,z>$, so if $k=0$, this is equal to $<0,z>$.

This completes the proof in all the cases.\end{proof}

\

2.7. {\it $\leq_P$ and $\oplus$}

\begin{lemma} If $<k,x>\leq_P <l,y>$ and $<m,z> \leq_P <n,w>$, then 
$$<k,x>\oplus <m,z> ~ \leq ~ <l,y>\oplus <n,w>.$$
\end{lemma}
\begin{proof} There are four cases:

a) $(x,z)$ and $(y,w)$ both in Case 1.

b) $(x,z)$ in Case 1, $(y,w)$ in Case 2.

c) $(x,z)$ in Case 2, $(y,w)$ in Case 1.

d) $(x,z)$ in Case 2, $(y,w)$ in Case 2.

Now suppose the hypothesis of the lemma, i.e. $k\leq l, x\leq y, m\leq n, z\leq w$. 

a) $<k,x> \oplus <m,z>=<k+m,x+z)$ and 
$$(l,y)\oplus (n,w)=(l+n,y+w),$$
and the required inequality follows from the four inequalities of the hypothesis and Axioms 1-10.

b) $<k,x>\oplus <m,z>=<k+m,x+z>$, and 
$$<l,y>\oplus <n,w>=<l+n+1,y\dotminus (\tau\dotminus w)>,$$
and the required inequality follows since $k+m\leq l+n+1$.

c) $<k,x>\oplus <m,z>=(k+m+1,x\dotminus (\tau\dotminus z))$, and 
$$<l,y>\oplus <n,w>=<l+n,y+w>.$$
But since $y\geq x$, and 
$w\geq z$, we have $y+w=\tau$, contradiction. This case does not occur.

d) $<k,x>\oplus <m,z>=<k+m+1,x\dotminus (\tau\dotminus z>$, and 
$$<l,y>\oplus <n,w>=<l+n+1,y\dotminus (\tau\dotminus w)>.$$ 
But now $x\leq y$, and (by Lemma \ref{cor1}) $\tau\dotminus z\geq \tau\dotminus w$, so 
$$x\dotminus(\tau\dotminus z)\leq y-(\tau\dotminus z)\leq y-(\tau\dotminus w)=y\dotminus (\tau\dotminus w)$$
(we are in Case 2), giving the result.\end{proof}

\

2.8. {\it Associativity.} Verifying this is the most tedious task of all. We can profit a bit from having already proved commutativity. We need to prove:

$$(1) \ \ \ \ \ \ \ \ <k,x>\oplus (<l,y>\oplus <m,z>) =$$
$$(2) \ \ \ \ \ \ \ \ (<k,x>\oplus <l,y>)\oplus <m,z>=$$
$$(3) \ \ \ \ \ \ \ \ <m,z>\oplus (<l,y>\oplus <k,x>).$$
So we should calculate $(1)$ via, firstly, Case of $(y,z)$ and then Case of $x$ and right-hand coordinate of $<l,y>\oplus <m,z>$. Then do same for $(3)$, switching $x$ and $z$. 

The simplest situation for $(1)$ is:

\

{\it Situation} 1. $(y,z)$ in Case 1 and then right-hand coordinate of $<l,y>\oplus <m,z>$ is in Case 1 with $x$.
So calculation gives $y+z+x <\tau$ and value of $(1)$ is 
$$<k+l+m,x+y+z>.$$
This is exactly the same as we get from 
$(3)$ assuming $(y,x)$ in Case 1, and so is $z$ with right-hand coordinate of $<l,y>\oplus <k,x>$, and so we verify one instance of associativity, when $x+y+z<\tau$.

\

{\it Situation} 2. $(y,z)$ in Case 1, and $x$ in Case 2 with right-hand coordinate of $<l,y>\oplus <m,z>$. So $y+z<\tau$ but $x+(y+z)=\tau$. Then value of $(1)$ is 
$$<k+l+m+1,x\dotminus (\tau\dotminus (y+z))>$$
(and bear in mind that $x\dotminus (\tau\dotminus (y+z))=(y+z)\dotminus (\tau\dotminus x)$).

With the same assumptions on $x,y,z$ we try to calculate the value of $(3)$,i.e. of $<m,z>\oplus (<l,y>\oplus <k,x>)$. 
From the preceding we have $y+z<\tau$ and $x+(y+z)=\tau$. 

Now we try to calculate $<l,y>\oplus <k,x>$.

\

{\it Subcase} 1. $y+x<\tau$. Then
$$<l,y>\oplus <k,x>=<L+k,y+x>$$
and since $x+(y+z)=\tau$ we have $z+(x+y)=\tau$, whence Case 2 for $z$ and $y+x$, whence
$$<m,z>\oplus (<l,y>\oplus <k,x>)=<m+l+k+1,z\dotminus (\tau\dotminus (x+y))>.$$

Now, Axiom 11 gives that if $$x+y+z=\tau,$$ 
and $y+z<\tau$ and $y+x<\tau$, then 
$$x\dotminus (\tau\dotminus (y+z))=z\dotminus(\tau\dotminus (x+y)),$$
so we have $(1)=(3)$ in this subcase.

\

{\it Subcase} 2. $y+x=\tau$. Then
$$<l,y>\oplus <k,x>=<l+k+1,x\dotminus(\tau\dotminus y)>= <l+k+1,y\dotminus(\tau \dotminus x)>$$

\

{\it Subcase} 2.1. $z+(y\dotminus (\tau \dotminus x))<\tau$. So $(3)$ is
$$<m+l+k+1,z+(y\dotminus (\tau \dotminus x))>$$
and by Axiom 12 we have $(1)=(3)$ in this subcase.

\

{\it Subcase} 2.2. $z+(y\dotminus (\tau\dotminus x))=\tau$. By Axiom 13 this is impossible.

\

{\it Situation} 3. $(y,z)$ in Case 2, and $x$ in Case 1 with right-hand coordinate of $<l,y>\oplus <m,z>$. 

So $y+z=\tau$, whence 
$$<l,y>\oplus <m,z>=<l+m+1,y\dotminus (\tau\dotminus z)>=<l+m+1,z\dotminus (\tau\dotminus y)>$$

Then $x+((y\dotminus (\tau\dotminus z))<\tau$, and so 
$$<k,x>\oplus (<l,y>\oplus <m,z>)=<k+l+m+1,x+(y\dotminus (\tau \dotminus z))>$$
which equals the value of $(1)$.

With the same assumptions on $(x,y,z)$ we try to compute $(3)$. The assumptions $y+z=\tau$ and 
$x+(y\dotminus (\tau \dotminus z))<\tau$.

\

{\it Subcase} 1. $y+x<\tau$. Then $<l,y>\oplus <k,x>=<l+k,y+x>$.

\

{\it Subcase} 1.1. $z+(y+x)<\tau$. This is impossible, since $y+z=\tau$ is assumed.

\

{\it Subcase} 1.2. $z+(y+x)=\tau$. 

So $(3)$ equals 
$$<k+l+m+1,z\dotminus (\tau\dotminus (y+x))>=<k+l+m+1,(y+x)\dotminus (\tau\dotminus z)>$$
which equals (1) by Axiom 11.

\

{\it Subcase} 2. $y+x=\tau$. Then 
$$<l,y>\oplus <k,x>=<l+k+1,y\dotminus (\tau\dotminus x)>=<l+k+1,x\dotminus (\tau \dotminus y)>.$$

{\it Subcase} 2.1. $z+(y\dotminus (\tau\dot x))<\tau$. 

So $(3)$ equals
$$<k+l+m+1,(x\dotminus (\tau\dotminus y))+z>=<k+l+m+1,(y \dotminus (\tau \dotminus x))+z>$$
which equals $(1)$ by Axiom 14. 

\

{\it Subcase} 2.2. $z+(y\dotminus (\tau\dotminus x))=\tau$. 

But this together with the assumptions of the situation contradicts Axiom ?. So it does not occur.

\

{\it Situation} 4. $(y,z)$ in Case 2 and $x$ in Case 2 with right-hand coordinate of $<l,y>\oplus <m,z>$. 

So $y+z=\tau$
whence
$$<l,y>\oplus <m,z>=<l+m+1,y\dotminus (\tau \dotminus z)>=<l+m+1,z\dotminus (\tau \dotminus y)>.$$
Then $x+(y\dotminus (\tau \dotminus z))=\tau$, so 
$$<k,x>\oplus (<l,y>\oplus <m,z>)=<k+l+m+2,x\dotminus (\tau \dotminus (y\dotminus (\tau \dotminus z)))>=<k+l+m+2,(y\dotminus (\tau \dotminus z)) \dotminus (\tau \dotminus x)>,$$
giving the value of $(1)$.

With the same assumptions on $(x,y,z)$, we try to compute $(3)$. 

The assumptions are $y+z=\tau$ and $x+(y\dotminus (\tau \dotminus z))=\tau$.

\

{\it Subcase} 1. $y+x<\tau$. 

But this is inconsistent with Axiom 6 since $x+(y\dotminus (\tau \dotminus z))=\tau$. 

\

{\it Subcase} 2. $y+x=\tau$. So
$$<l,y>\oplus <k,x>=<l+k+1,y\dotminus (\tau \dotminus x))=<l+k+1,x\dotminus (\tau \dotminus y)>.$$

\

{\it Subcase} 2.1. $z+(y\dotminus (\tau\dotminus x))<\tau$. 

This is inconsistent by Axiom 15, since 
$$x+(y\dotminus (\tau \dotminus z))=\tau$$

\

{\it Subcase} 2.2. $z+(y\dotminus (\tau\dotminus x))=\tau$. 

Then $(3)$ equals 
$$<k+l+m+2,(y\dotminus (\tau \dotminus x))\dotminus (\tau \dotminus z)$$ 
and by Axiom 16 this equals 
$$<k+l+m+2,(y\dotminus (\tau\dotminus z)) \dotminus (\tau \dotminus x)>$$
so $(1)=(3)$.

This concludes the proof that $\oplus$ is associative.

\

2.9. {\it Cancellation.} In the preceding we have established that $(P,\leq_P,\oplus)$ is a commutative ordered monoid with $0$ as least element. In 2.7 we show (en passant) that
$<k,x>\leq <k,x>\oplus <l,y>$ for any $<k,x>$, $<l,y>$.

It remains to prove cancellation or relative complementation, namely:

\begin{lemma} If $<k,x> \leq <l,y>$ then there exists $<m,z>$ with
$$<k,x>\oplus <m,z>=<l,y>.$$
\end{lemma}
\begin{proof} Assume $<k,x>\leq <l,y>$. Then $k\leq l$. If $k=l$, then $x\leq y$. So take $m=0$ and $z=y\dotminus x$. 

If, however, $k<l$, there are two cases.

\

{\it Case} 1. $x\leq y$. So take $m=l-k$ and $z=y\dotminus x$.

\

{\it Case} 2. $y<x$. 

\

We need $<m,z>$ so that 
$$<l,y>=<k,x>\oplus <m,z>,$$
and as usual the Case 1/Case 2 distinction on $x,z$ 
intervenes. But now $z$ is the unknown, with $x,y$ given.

Suppose $z$ can be found in Case 1. So $x+z<\tau$, and then $y=x+z$, so $y\geq x$, contradicting our assumption. 

Thus $z$ can be found, if at all, only in Case 2, and then 
$$y=x \dotminus (\tau \dotminus z)=z\dotminus (\tau \dotminus x),$$ 
so $\tau \dotminus z=x \dotminus y$. 

So $z=\tau \dotminus(x\dotminus y)$. What about $m$? We want 
$$<l,y>=<k+m+1,x\dotminus(\tau \dotminus z)>,$$
so we need only 
$$l=k+m-1,$$
so $m=(l-k)-1\geq 0$.\end{proof}
This completes the proof of Theorem \ref{toag}.\end{proof}

\

2.10. We are going to use Theorem \ref{toag} only for discretely ordered TOAG, in fact only the 
$[0,\tau]$ coming from models of Presburger arithmetic. 

\

2.11. {\it Presburger truncated ordered abelian groups.} In our applications we take Presburger arithmetic (cf. \cite{presburger}) to be formulated in the language of ordered groups with a distinguished constant $1$ (to denote the least positive element). We generally drop the distinction between the group $\Gamma$ and its non-negative part.

Note that $0\neq 1$ in models of Presburger.

We now consider Presburger truncated ordered Abelian groups, i.e. truncated ordered Abelian groups of the form $[0,\tau]$ with distinguished element $1$, the least positive element (we do not insist that $1<\tau$), which are truncations of models of Presburger. 

\begin{thm}\label{thm2} A truncated ordered Abelian group $[0,\tau]$ with least positive element $1$ is a Presburger truncated ordered Abelian group if and only if it satisfies the following conditions:
\begin{itemize}
\item $[0,\tau]$ is discretely ordered and every positive element is a successor,
\item For each positive integer $n$ and each $x$ in $[0,\tau]$ there is a $y$ in $[0,\tau]$ and an integer $m<n$ such that 
$$x=ny+m=(y+\dots+y)+(1+\dots+1).$$
\end{itemize}
\end{thm}

({\it Note.} When $m=0$ in $\Z$, $m=0$ in $[0,\tau]$ by definition.)

\begin{proof} Necessity is clear from the axioms of Presburger. 

For sufficiency, we argue as follows. Suppose $[0,\tau]$ 
satisfies Conditions (1) and (2) (and has $1$). Build $P$ as in the proof of Theorem \ref{toag}. Clearly $1$ is the least positive element of $P$ and $P$ is discretely ordered. Let $<k,x>$ be a nonzero element of $P$, so $k\in \{0,1,2,\dots\}$ and $x \in [0,\tau)$. If 
$x\neq 0$,
$$<k,x>=<k,x\dotminus 1>+<0,1>,$$
so $<k,x>$ is a successor. If $x=0$, and $k\neq 0$, $<k,x>$ is successor of $<k-1,\tau-1>$.

So $P$ is discretely ordered, and every positive element is a successor. 

To get the Euclidean division results, fix a positive integer $n$ and some $<k,x>$ in $P$. Let $k=na+b$, for non-negative integers $a,b$ with $b<n$. Now
$$<k,x>=<k,0>\oplus <0,x>=<na,0>\oplus <b,0>\oplus <0,x>.$$
Now
$$<na,0>=\underbrace{<a,0>\oplus \dots \oplus <a,0>}_{n \ \text{times}},$$
Also, if $b>0$, 
$$<b,0>=\underbrace{<1,0>\oplus \dots \oplus <1,0>}_{b \ \text{times}}$$
and 
$$<1,0>=<0,\tau-1>+1,$$
and if $1<\tau$ (other case trivial)
$$\tau-1=nc+d,$$
for some $c\in [0,\tau-1]$, and $0\leq d<n$ (with usual conventions about multiplication by $m$).

Finally, $<0,x>$ is also of the form
$n\gamma+\delta$, with $0\leq \delta<n$ via Condition 2 for $[0,\tau]$.

Thus $<k,x>$ is congruent "modulo" $n$ to an integer less than $n$, and we are done.\end{proof}

\section{\bf Elementary theories of Presburger truncated ordered abelian groups}

3.1. When is $[0,\tau] \cong [0,\mu]$ if both are Presburger truncated ordered Abelian groups? Clearly the answer comes, by Theorem \ref{thm2}, from 
an answer to the question: what are the pure 1-types for Presburger arithmetic? 

The answer to this is well-known \cite{presburger}. Namely, the pure $1$-type of an element $x$ in $P$, the non-negative part of a model of Presburger, is determined by

(A) Whether or not $x\in \{0,1,2,\dots\}$,

(B) The remainder of $x$ modulo $n$ for each positive integer $n$.

\

3.2. \begin{thm}\label{thm3} The elementary theory of a Presburger truncated ordered Abelian groups $[0,\tau]$ is determined by the Presburger $1$-type of $\tau-1$ 
(the penultimate element of $[0,\tau]$), i.e. by

(A) Whether $\tau -1 \in \{0,1,2,\dots\}$

(B) The congruence class of $\tau-1$ modulo $n$ for each positive integer $n$.

Moreover, any Presburger $1$-type can occur for some truncated ordered Abelian groups $[0,\tau]$.\end{thm}
\begin{proof} Immediate from the preceding.\end{proof}

\


\bibliographystyle{acm}
\bibliography{bibadeles}

\begin{thebibliography}{1}

\bibitem{elem-prod}
{\sc D'Aquino, P., and Macintyre, A.}
\newblock Commutative unital rings elementarily equivalent to prescribed
  product products.
\newblock {\em Preprint\/}.

\bibitem{PDAJM}
{\sc D'Aquino, P., and Macintyre, A.}
\newblock Zilber problem on residue rings of models of arithemtic, part 1: the
  prime power case.
\newblock {\em Preprint\/} (2019).

\bibitem{elem-rest}
{\sc Derakhshan, J., and Macintyre, A.}
\newblock On rings elementarily equivalent to restricted products.
\newblock {\em In preparation\/}.

\bibitem{hiranouchi-survey}
{\sc Hiranouchi, T.}
\newblock Ramification of truncated discrete valuation rings: a survey.
\newblock In {\em Algebraic number theory and related topics 2008}, RIMS
  K\^{o}ky\^{u}roku Bessatsu, B19. Res. Inst. Math. Sci. (RIMS), Kyoto, 2010,
  pp.~35--43.

\bibitem{hiranouchi-ext}
{\sc Hiranouchi, T., and Taguchi, Y.}
\newblock Extensions of truncated discrete valuation rings.
\newblock {\em Pure Appl. Math. Q. 4}, 4, Special Issue: In honor of
  Jean-Pierre Serre. Part 1 (2008), 1205--1214.

\bibitem{presburger}
{\sc Presburger, M.}
\newblock On the completeness of a certain system of arithmetic of whole
  numbers in which addition occurs as the only operation.
\newblock {\em Hist. Philos. Logic 12}, 2 (1991), 225--233.
\newblock Translated from the German and with commentaries by Dale Jacquette.

\end{thebibliography}

\end{document}